\documentclass{amsart}

\newtheorem{thma}{Theorem}
\theoremstyle{remark}
\newtheorem{remark}[thma]{Remark}
\newtheorem{ex}[thma]{Example}
\theoremstyle{plain}
\newtheorem{lemma}[thma]{Lemma}
\newcommand{\supp}{\operatorname{supp}}
\newcommand{\diag}{\mathop{\mathrm{diag}}}
\newcommand{\diam}{\mathop{\mathrm{diam}}}
\newcommand{\trace}{\operatorname{trace}}
\newcommand{\ring}{\ensuremath{\mathcal{R}}}
\newcommand{\field}{\ensuremath{\mathbb{F}}}
\newcommand{\vecspace}[1]{\ensuremath{\mathcal{#1}}}
\newcommand{\mat}[2]{\ensuremath{{#1}^{#2\times #2}}}

\newcommand{\free}{\mathbf{F}}

\title[Fixed point theorems for nc functions]{Fixed point theorems
   for noncommutative functions}
\author[G. Abduvalieva]{Gulnara Abduvalieva}
\address{Department of Mathematics \\
Drexel University\\
3141 Chestnut Str.\\
  Philadelphia, PA, 19104}
\email{gka26@drexel.edu}

\author[D.S.~Kaliuzhnyi-Verbovetskyi]{Dmitry S.
Kaliuzhnyi-Verbovetskyi}
\address{Department of Mathematics \\
Drexel University\\
3141 Chestnut Str.\\
  Philadelphia, PA, 19104}
\email{dmitryk@math.drexel.edu}
\date{}
\thanks{The second author was partially supported by the NSF grant DMS 0901628 and by the US--Israel BSF grant 2010432.}

\begin{document}

\maketitle

\begin{abstract}
We establish a fixed point theorem for mappings of square matrices
of all sizes which respect the matrix sizes and direct sums of
matrices. The conclusions are stronger if such a mapping also
respects matrix similarities, i.e., is a noncommutative function.
As a special case, we prove the corresponding contractive mapping
theorem which can be viewed as a new version of the Banach Fixed
Point Theorem. This result is then applied to prove the existence
and uniqueness of a solution of the initial value problem for ODEs
in noncommutative spaces. As a by-product of the ideas developed
in this paper, we establish a noncommutative version of the
principle of nested closed sets.
\end{abstract}

\section{Introduction}\label{sec:intro}
The theory of noncommutative (nc) functions has its origin in the
articles of Joseph L. Taylor \cite{JLTaylor,JLTaylor-nc}. It was
further developed by D.-V. Voiculescu \cite{Voic1,Voic2} in his
fundamental work in free probability. Based on pioneering ideas of
J. L. Taylor, the second author and Victor Vinnikov \cite{KV-VV}
developed the nc difference-differential calculus and used it for
studying various questions of nc analysis, in particular,
extending the classical (commutative) theory of analytic functions
to a nc setting. A special case of nc rational functions, which is
important for applications in optimization and control, where
matrices are natural variables and the problems are
dimension-independent (see \cite{Helton,HMcCV,HMcCPV} for a
detailed discussion), can be studied independently --- see
\cite{KVV1,KVV2}. The theory of nc rational functions is motivated
by and useful in nc semialgebraic geometry --- see, e.g.,
\cite{Helton0,HMcC,HP,HKMcC}. We also mention the works of
Helton--Klep--McCullough \cite{HKMcC1,HKMcC2}, of Popescu
\cite{Pop1,Pop2,Pop3}, and of Muhly--Solel \cite{MS} on various
aspects of nc function theory. The goal of the present paper is to
establish a certain type of fixed point theorems as a useful tool
in nc analysis.

We provide the reader with some basic definitions from
\cite{KV-VV}.
 Let $ \mathcal{R} $ be a unital ring. For a bi-module $\mathcal{M}$ over
$\mathcal{R},$ we
 define the \emph{nc space over  $ \mathcal{M} $},
\begin{equation}\label{eq:ncspace}
 \mathcal{M}_{\rm nc}:=\coprod_{n=1}^{\infty}  \mathcal{M}^{n \times n}.
\end{equation}

 A subset $ \Omega  \subseteq  \mathcal{M}_{\rm nc} $ is called a \emph{nc set} if it is closed under direct sums;
that is, denoting $\Omega_{n}=\Omega \cap \mathcal{M}^{n\times n}$, we have
\begin{equation}\label{eq:ncset}
   X \in \Omega_{n},   Y \in \Omega_{m}   \Longrightarrow X \oplus Y := \left [ \begin{array}{cc}
X& \mathbf{O}_{n\times m}
\\ \mathbf{O}_{m\times n} &Y
  \end{array} \right ] \in\Omega_{n+m},
\end{equation}
 where $\mathbf{O}_{p\times q}$ denotes the $p\times q$ matrix whose all entries are 0.

Notice that  matrices over $ \mathcal{R}  $  act from the right and from the left on matrices over $\mathcal{M}$
by the standard rules of matrix multiplication: if $ T \in \mathcal{R}^{r \times p} $ and  $ S \in
\mathcal{R}^{p\times s} $, then for $ X \in \mathcal{M}^{p \times p} $ we have
 \[TX \in  \mathcal{M}^{r \times p}, \quad       XS \in  \mathcal{M}^{p \times s}. \]

 In the special case where
$\mathcal{M}=\mathcal{R}^{d}$, we identify matrices over
$\mathcal{M}$  with  $ d $-tuples of matrices over $\mathcal{R}$
:\[ (\mathcal{R}^{d})^{p\times q}\cong( \mathcal{R}^{p\times q})^{d}
.\] Under this identification, for $d$-tuples  $X=(X_{1}, \ldots, X_{d}) \in
(\mathcal{R}^{n\times n})^{d} $ and $Y=(Y_{1}, \ldots, Y_{d}) \in
(\mathcal{R}^{m\times m})^{d} $ their direct sum has the form
\[  X \oplus Y = \left ( \left [ \begin{array}{cc} X_{1}& \mathbf{O}_{n\times m}\\ \mathbf{O}_{m\times n} &Y_{1}
 \end{array} \right ], \ldots,
\left [ \begin{array}{cc} X_{d}& \mathbf{O}_{n\times m}\\ \mathbf{O}_{m\times n} &Y_{d} \end{array} \right ]
\right)
 \in (\mathcal{R}^{(n+m)\times(n+m)})^{d}; \] and for a $d$-tuple $X=(X_{1}, \ldots, X_{d}) \in
(\mathcal{R}^{p\times p})^{d} $ and matrices $T \in \mathcal{R}^{r\times p}$,
   $S \in \mathcal{R}^{p\times s}$,
\[TY=(TY_{1}, \ldots, TY_{d}) \in (\mathcal{R}^{r\times p})^{d},\qquad XS=(X_{1}S, \ldots, X_{d}S)
 \in (\mathcal{R}^{p\times s})^{d};  \] that is, $T$ and $S$ act on $d$-tuples of matrices componentwise.

 Let $\mathcal {M}$ and $\mathcal{N}$ be bi-modules over
$ \mathcal{R}$, and let $ \Omega \subseteq \mathcal{M}_{\rm nc}$ be a nc set. A mapping \[  f: \Omega \to
\mathcal{N}_{\rm nc}\] with the property that $f(\Omega_{n}) \subseteq \mathcal{N}^{n\times n}$, $n=1,2,\ldots$,
is called  a \emph{nc function} if $f$ satisfies  the following two conditions:
\begin{equation}\label{eq:dirsums}   f \ \emph{respects\ direct\ sums:\ }
f(X\oplus Y) = f(X) \oplus f(Y),\quad X, Y \in \Omega; \end{equation}
\begin{multline}
\label{eq:sim}  f \ \emph{respects\ similarities:}\ {\rm if\ } X
\in \Omega_{n}\ {\rm and}\  S \in \mathcal{R}^{n\times n}\ {\rm
is}\ {\rm invertible\ }  \\{\rm with\ }  SXS^{-1} \in \Omega_{n},\
{\rm then}\
  f(SXS^{-1}) = Sf(X)S^{-1}, \end{multline}
or, equivalently, satisfies the single condition:
\begin{multline}\label{eq:intertw}  f\ respects\ intertwinings:\
{\rm if}\ X \in \Omega_{n}, Y \in \Omega_{m},\ {\rm and}\ T \in \mathcal{R}^{n\times m}\\
{\rm are\ such\ that\ } XT = TY,\ {\rm then}\ f(X)T = Tf(Y).
\end{multline}

Condition \eqref{eq:intertw} was used by J. L. Taylor in the case
where $ \mathcal{M} = \mathbb{C}^{d},$ together with an additional
assumption of analyticity of $f(X)$ as a function of matrix
entries $ ( X_{i})_{j, k},\ i = 1, \ldots, d;\ j,  k = 1, \ldots,
n, $ for every $ n\in \mathbb{N}$ --- see \cite{JLTaylor-nc}.

Notice a certain discrepancy in our terminology (inherited from
\cite{KV-VV}): a set is nc if it respects direct sums, while a
function is nc if it respects both direct sums and similarities.
Nevertheless, we prefer to keep it this way, setting the minimal
assumptions under which the theory of nc functions starts
revealing its phenomena.

\begin{ex}\label{ex:1}
Let $$p=\sum_{w\in\free_d\colon |w|\le m}p_wz^w$$ be a nc
polynomial, where $\free_d$ is the free semigroup on $d$
generators (an alphabet) $g_1$, \ldots, $g_d$, with the unit
element $\emptyset$ (the empty word); for a word $w=g_{i_1}\cdots
g_{i_k}\in\free_d$ and a $d$-tuple of noncommuting indeterminates
$z=(z_1,\ldots,z_d)$, we use the notation $z^w:=z_{i_1}\cdots
z_{i_k}$ (and $z^\emptyset:=1$); the length of the word
$w=g_{i_1}\cdots g_{i_k}$ is $|w|:=k$ (and $|\emptyset|:=0$);
$p_w\in\ring$ for every $w\in\free_d$. Then $p$ can be viewed as a
nc function $p\colon\Omega=(\ring^d)_{\rm nc}\to\ring_{\rm nc}$:
for a $d$-tuple $X=(X_1,\ldots,X_d)$ of $n\times n$ matrices over
$\ring$, we set
$$p(X):=\sum_{w\in\free_d\colon |w|\le m}p_wX^w,$$
where $X^w:=X_{i_1}\cdots X_{i_k}$ for $w=g_{i_1}\cdots g_{i_k}\in\free_d$ and $X^\emptyset:=\mathbf{I}_n$ (the
$n\times n$ identity matrix).
\end{ex}
\begin{ex}\label{ex:2}
Let $\ring=\mathbb{C}$, $\mathcal{M}=\mathbb{C}^d$, $\mathcal{N}=\mathbb{C}$. Let
$\Omega\subseteq\mathcal{M}_{\rm nc}$ be the nc set of $d$-tuples $X=(X_1,\ldots,X_d)$ of square matrices of the
same size such that the spectral radius of $X_1+\cdots+X_d$ is strictly less than 1. Then
$f\colon\Omega\to\mathcal{N}_{\rm nc}$ defined by
\begin{multline*}
f(X):=(\mathbf{I}_n-X_1-\cdots-X_d)^{-1}=\sum_{j=0}^\infty(X_1+\cdots+X_d)^j=\sum_{j=0}^\infty\sum_{w\in\free_d\colon
|w|=j}X^w
\end{multline*}
is a nc function.
\end{ex}
In the next two examples, we present some mappings of matrices
which respect matrix sizes, but fail to be nc functions because
one of the two conditions, \eqref{eq:dirsums} or \eqref{eq:sim},
is not satisfied.

\begin{ex}\label{ex:3}
Let $\ring$ be commutative, and let $f\colon\Omega=\ring_{\rm
nc}\to\ring_{\rm nc}$ be defined by
$$f(X):=(\det X)\mathbf{I}_n,\quad X\in\mat{\ring}{n},\ n\in\mathbb{N}.$$
Then, clearly, $f$ respects the matrix size and similarities,
i.e., $f(\Omega_n)\subseteq\mat{\ring}{n}$, $n\in\mathbb{N}$, and
\eqref{eq:sim} holds. However, $f$ does not respect direct sums,
i.e., \eqref{eq:dirsums} does not always hold. Thus, $f$ is not a
nc function. Notice that replacing $\det X$ by $\trace X$, i.e.,
defining
$$f(X):=(\trace X)\mathbf{I}_n,\quad X\in\mat{\ring}{n},\ n\in\mathbb{N},$$
would bring the same conclusions.
\end{ex}
\begin{ex}\label{ex:4}
Let $\ring=\mathcal{M}=\mathcal{N}=\field$, where
$\field=\mathbb{R}$ or $\field=\mathbb{C}$. For every scalar $t$,
define $f_t\colon\mathcal{M}_{\rm nc}\to\mathcal{N}_{\rm nc}$ by
$$f_t(X):=t\widehat{X},$$
where $\widehat{X}$ is a square matrix over $\field$ which has the same size and the same main diagonal as $X$,
and whose all off-diagonal entries are zeros. Then, clearly, $f_t$ respects direct sums, i.e., \eqref{eq:dirsums}
holds, however $f_t$ does not respect similarities, i.e., \eqref{eq:sim} does not always hold, for $t\neq 0$. For
example, let
$$X=\begin{bmatrix}
1 & 1\\
1 & -1
\end{bmatrix},\quad S=\begin{bmatrix}
1/\sqrt{2} & 1/\sqrt{2}\\
-1/\sqrt{2} & 1/\sqrt{2}
\end{bmatrix}.$$
Then $$\widehat{SXS^{-1}}=\begin{bmatrix}
1 & 0\\
0 & -1
\end{bmatrix}\neq\begin{bmatrix}
0 & -1\\
-1 & 0
\end{bmatrix}=S\widehat{X}S^{-1}.$$
Thus, $f_t$ is not a nc function unless $t=0$.
\end{ex}

\section{The results} \label{sec:results}
A part of our main results is valid under much more general
assumptions on the sets of matrices and the corresponding
functions. If $\mathcal{S}$ is a set, then we will write
$\mat{\mathcal{S}}{n}$ for a set of $n\times n$ matrices with
entries from $\mathcal{S}$, where we consider matrices just as
arrays, i.e., not assuming any algebraic structure on them. We
will also extend the notation of \eqref{eq:ncspace} to an
arbitrary set $\mathcal{S}$ in the place of $\mathcal{M}$.

Fix some element $O\in\mathcal{S}$. We define direct sums of matrices from $\mathcal{S}_{\rm nc}$ in the same way
as in \eqref{eq:ncset} except that $\mathbf{O}_{p\times q}$ is understood now as the $p\times q$ matrix whose all
entries are equal to $O$.

\begin{thma} \label{thm:thfixedp-1}
Let $\mathcal{S}$ be a set, $O\in\mathcal{S}$, and let $\Omega
\subseteq \mathcal{S}_{\rm nc}$ respect direct sums of matrices.
Define $ \supp{ \Omega} := \{ n \in \mathbb{N}: \Omega_{n} \neq
\emptyset \}.$ Let $f: \Omega \rightarrow \Omega$ satisfy
                                          \begin{itemize}
\item $f(\Omega_{n}) \subseteq \Omega_{n}$, $ n \in \supp{\Omega};$ \item   $f$ respects  direct  sums:
$f(X\oplus Y) = f(X) \oplus f(Y),\quad X, Y \in \Omega;$ \item For every $n \in \supp{\Omega}$ the mapping $f
|_{\Omega_{n}}$ has a unique fixed point, $X_{*n}$.
\end{itemize}
Let  $d = {\rm gcd}\{ n: n \in \supp{\Omega}\}$. Then
\begin{enumerate}
\item[1.] There exists $X_{*}\in \mathcal{S}^{d\times d}$ such that
\begin{equation}\label{eq:fixedp} X_{* n} = \bigoplus_{\alpha = 1}^{n/d}X_{*}, \quad n\in \supp{\Omega}.\end{equation}
\item[2.]  If, moreover, $\mathcal{S}=\mathcal{M}$ is a bi-module
over $\ring$, $O=0\in\mathcal{M}$, and $f$  is a nc function, then
there exists  a nc set $ \tilde{\Omega} \supseteq \Omega $
 with $\supp{\tilde{\Omega}} = \mathbb{N}d$ and a nc function $\tilde{f}:\tilde{ \Omega} \rightarrow \tilde{ \Omega}$
such that
\begin{itemize}
\item  $\tilde{f} |_{\Omega} = f;$
\item For every $n \in \mathbb{N}d$ the mapping $\tilde{f} |_{\tilde{\Omega}_{n}}$ has a unique fixed point
\[  X_{* n} = \bigoplus_{\alpha = 1}^{n/d}X_{*}.\]
\end{itemize}
\end{enumerate}
\end{thma}

Theorem \ref{thm:thfixedp-1} establishes the following general
principle. Given a mapping on matrices which respects
 matrix sizes and direct sums of matrices, and given that this mapping has a unique fixed point in each matrix
dimension, all these fixed points arise as multiple copies of a
single matrix $X_*$; moreover, if the mapping is a nc function,
then it can be extended to a nc function whose domain includes the
matrix $X_*$. This principle can
 be used further to generalize \emph{any fixed point theorem where the fixed point is unique} to the
noncommutative setting; in Theorem \ref{thm:thfixedp-2}, we will
present a nc version of the Banach Fixed Point Theorem.

\begin{remark} A statement analogous to part 2 of Theorem \ref{thm:thfixedp-1} for a mapping $f$ which
respects direct sums but  not necessarily respects similarities is
false, as the
 following example shows.
\end{remark}
\begin{ex}Let  $\mathcal{M} = \mathcal{R}=\mathbb{Z}_{2}(=\mathbb{Z}/2 \mathbb{Z})$, and let $\Omega$ be defined
as follows:
 $$\supp{\Omega} = \{ 2n + 5m, \quad n\in \mathbb{N}, m\in \mathbb{N}\},$$
$\Omega_{2} = \{ \mathbf{O}_{2\times2},\mathbf{I}_{2}\},$
$ \Omega_{5} = \{ \mathbf{O}_{5\times5}, \mathbf{I}_{5},
 \diag[0, 0, 1, 0, 0]\}$, and $\Omega_{2n + 5m}$
is a finite set of matrices which are direct sums, in any possible
order, of matrices from $\Omega_{2}$ and $\Omega_{5}$. Clearly, $\Omega $ is a nc set over $ \mathbb{Z}_{2}.$
Define
$$ f(\mathbf{O}_{2 \times 2}) = \mathbf{I}_{2} = f(\mathbf{I}_{2}),$$
$$ f(\mathbf{O}_{5 \times 5}) = \mathbf{I}_{5} = f(\mathbf{I}_{5}),$$
$$f(\diag[0, 0, 1, 0, 0]) = \mathbf{O}_{5 \times 5}.$$
Then $f$ is uniquely
determined. Indeed, elements of $\Omega$ are diagonal matrices, and the sequence 0, 0, 1, 0, 0 appears on the
diagonal of a matrix $X$ from $\Omega$ if and only if this sequence belongs to a copy of the matrix
$A = \diag[0, 0, 1, 0, 0] \in \Omega_{5},$ i.e. $X = Y \oplus A \oplus Z,$ with $Y, Z \in \Omega$. Then
$$ f(X) = f(Y) \oplus f(A) \oplus f(Z) = f(Y) \oplus \mathbf{O}_{5 \times 5} \oplus f(Z).$$
If $Y$ and $Z$ have no copies of $A$ on the diagonal, then
$$f(X) = \mathbf{I}\oplus \mathbf{O}_{5\times 5}\oplus \mathbf{I}.$$
Otherwise, applying the same argument for $Y$ (and $Z $) in the place of $X$,
 we eventually obtain that $f(X)$ is uniquely determined and equal a diagonal matrix with
 $\mathbf{O}_{5 \times 5}$ at the same block diagonal spots as copies of $A$ in $X,$ and $1$
at the other diagonal spots. Suppose there exists a nc set $\tilde{\Omega}\supseteq \Omega$ such that
$\supp{\tilde{\Omega}} = \mathbb{N}$ (clearly, we have $d={\rm gcd}(2,5)=1$), and a mapping $\tilde{f}:\tilde{
\Omega} \rightarrow \tilde{\Omega}$ satisfying
\begin{itemize}
\item $\tilde{f}(\tilde{\Omega}_{n}) \subseteq \tilde{\Omega}_{n};$ \item   $\tilde{f}$ respects  direct  sums:
$\tilde{f}(X\oplus Y) =\tilde{f}(X) \oplus \tilde{f}(Y),\quad X, Y \in \tilde{\Omega};$ \item  $\tilde{f}
|_{\Omega} = f;$ \item For every $n \in \mathbb{N}$ the mapping $\tilde{f} |_{\tilde{\Omega}_{n}}$ has a unique
fixed point
$$  X_{* n} = \bigoplus_{\alpha = 1}^{n}X_{*}.$$
\end{itemize}
Then $\tilde{f}(\diag[0, 0, 1, 0, 0]) =f(\diag[0, 0, 1, 0, 0]) = \mathbf{O}_{5\times 5}. $ On the other hand,
since $f(\mathbf{I}_2)=\mathbf{I}_2$, we must have $1=\mathbf{I}_1\in\tilde{\Omega}_1$, $\tilde{f}(1)=1$, and
\begin{multline*}
\tilde{f}(\diag[0, 0, 1, 0, 0]) = \tilde{f}(\mathbf{O}_{2\times 2}\oplus 1 \oplus \mathbf{O}_{2 \times 2}) =
 \tilde{f}(\mathbf{O}_{2\times 2})\oplus \tilde{f}(1) \oplus \tilde{f}( \mathbf{O}_{2 \times 2})    \\
={f}(\mathbf{O}_{2\times 2})\oplus \tilde{f}(1) \oplus {f}( \mathbf{O}_{2 \times 2}) = \mathbf{I}_{2}\oplus 1
\oplus \mathbf{I}_{2} = \mathbf{I}_{5},
\end{multline*}
 i.e. we obtain a contradiction.
\end{ex}

We will need the following definition of a complex (real) operator
space (see \cite{ER,Ruan}) which gives rise to a natural topology
on a nc space. Let $\field$ be a field, $\field=\mathbb{C}$ or
$\field=\mathbb{R}$.
 A vector space $\vecspace{V}$ over $\field$ is called an \emph{operator space} if a
 sequence of Banach space norms $\|\cdot\|_n$ on $\mat{\vecspace{V}}{n}$, $n=1,2,\ldots$ is
defined so that the following two
 conditions hold:
 \begin{itemize}
     \item For every $n,m\in\mathbb{N}$, $X\in\mat{\vecspace{V}}{n}$ and
     $Y\in\mat{\vecspace{V}}{m}$,
     \begin{equation*}
 \| X\oplus Y\|_{n+m}=\max\{\|X\|_n,\| Y\|_m\};
     \end{equation*}
     \item For every $n\in\mathbb{N}$, $X\in\mat{\vecspace{V}}{n}$ and
     $S,T\in\mat{\field}{n}$,
     \begin{equation*}
 \| SXT\|_{n}\le\|S\|\,\|X\|_n\|T\|,
     \end{equation*}
    where $\|\cdot\|$ denotes the $(2,2)$ operator norm on
    $\mat{\field}{n}$.
 \end{itemize}

\begin{thma} \label{thm:thfixedp-2}
Let $\mathcal{S}$ be a set, $O\in\mathcal{S}$, and let $\Omega
\subseteq \mathcal{S}_{\rm nc}$ respect direct sums of matrices.
Suppose that $\Omega_{n}$ is a complete metric space with respect
to a metric $\rho_{n}$ for every $n \in \supp{\Omega}.$ Let
$f\colon \Omega \rightarrow \Omega$ satisfy
\begin{itemize}
\item $f(\Omega_{n}) \subseteq \Omega_{n}$, $ n \in \supp{\Omega};$ \item   $f$ respects  direct  sums:
$f(X\oplus Y) = f(X) \oplus f(Y),\quad X, Y \in \Omega;$ \item For every $n \in \supp{\Omega}$ there exists
$c_{n}\colon 0\le c_{n} <1$ so that
$$ \rho_{n}(f(X), f(Y)) \leq c_{n}\rho_{n}(X, Y), \quad X, Y \in \Omega_{n}.$$
 \end{itemize}
Let $d = {\rm gcd} \{n: n \in \supp{\Omega}\}.$ Then:
\begin{enumerate}
\item[1.] There exists $X_{*} \in \mathcal{S}^{d\times d}$ such that for every $n \in \supp{\Omega}$ the mapping
$f|_{\Omega_{n}}$ has  a unique fixed point $X_{*n} = \bigoplus_{\alpha = 1}^{n/d}X_{*}.$ \item[2.] Suppose
additionally that $\mathbb{F}$ is a field,  $\mathbb{F} = \mathbb{C}$ or $\mathbb{F} = \mathbb{R}$, $\mathcal{S}
= \mathcal{V}$ is an operator space over $\mathbb{F},$ so that for every $ n \in \supp{\Omega}$ one has
$\rho_{n}(X, Y) =  \| X - Y \|_{n}, \  X, Y \in \Omega_{n}$, $O=0\in\mathcal{V}$, and that $f$ is a nc function.
Then the conclusions of Theorem \ref{thm:thfixedp-1}, part 2, hold; moreover $\tilde{\Omega}$ and $\tilde{f}$ can
be chosen such that for every $n \in \mathbb{N}d:$
\begin{itemize}
\item $\tilde{\Omega}_{n}$ is a complete metric space with respect to the metric
$$\tilde{\rho_{n}}(X, Y) =  \| X - Y \|_{n}, \quad  X, Y \in \tilde{\Omega}_{n}$$ (which
extends the metric $\rho_{n}$ for $n \in \supp{\Omega})$; \item There exists $\tilde{c}_{n}\colon 0 \le
\tilde{c}_{n} < 1$   (obviously, $\tilde{c}_{n} \geq c_{n}$ for $ n \in \supp{\Omega}$)
 such that
\begin{equation}\label{eq:rel-contract}
\| \tilde{f}(X) - \tilde{f}(X_{* n}) \|_{n} \leq \tilde{c}_{n} \| X - X_{* n} \|_{n}, \quad X \in
\tilde{\Omega}_{n};
\end{equation}
 \item For an arbitrary $  X^{0} \in \tilde{\Omega}_{n},$ define $X^{j + 1} = f(X^{j}), \ j =
0, 1, \ldots.$ Then  $X_{* n} = \lim_{j \rightarrow \infty} X^{j}.$ Moreover,
\begin{equation}\label{eq:conv_rate}
 \| X^{j} - X_{* n} \|_{n} \  \leq (\tilde{c}_{n})^{j}
\| X^{0} - X_{*n} \|_{n} ,\quad j = 1, 2, \ldots.\end{equation}
 \end{itemize}
\end{enumerate}
\end{thma}

\begin{remark}\label{rem:open-question}
We do not know whether one can find  $\tilde{\Omega}$ and
$\tilde{f}$ such that for every $n \in \mathbb{N}d$ there exists
${c}_{n}^\circ\colon 0 \le {c}_{n}^\circ < 1$ satisfying
\[ \| \tilde{f}(X) - \tilde{f}(Y) \|_{n} \leq {c}_{n}^\circ \| X - Y \|_{n},\quad X, Y\in\tilde{\Omega},\]
and leave this as an open question.
\end{remark}
\begin{remark}\label{rem:quotient-space}
We now give a useful interpretation of part 2 of Theorem
\ref{thm:thfixedp-2}. First of all, in a general setting of part 1
of Theorem \ref{thm:thfixedp-2}, we define a relation $\sim$ on a
$\Omega\subseteq\mathcal{S}_{\rm nc}$ as follows: $X\sim Y$ if
both $X$ and $Y$ are direct sums of several copies of the same
matrix over $\mathcal{S}$. Observe that $\sim$ is an equivalence
relation on $\Omega$, and define the quotient set
$\widehat{\Omega}:=\Omega/\!\!\sim$. We will call the equivalence
class $\widehat{X}\in\widehat{\Omega}$ of $X\in\Omega$ a
\emph{noncommutative singleton}. Since $f\colon\Omega\to\Omega$
preserves the equivalence $\sim$, it gives rise to the quotient
mapping $\widehat{f}\colon\widehat{\Omega}\to\widehat{\Omega}$.
The conclusion of part 1 of Theorem \ref{thm:thfixedp-2} means
that the fixed points $X_{*n}$ are equivalent, and therefore the
mapping $\widehat{f}$ has a unique fixed point. However, the
assumptions of part 1 are too general to have a good
interpretation in terms of the mapping $\widehat{f}$. If one
strengthen them as in part 2, then one defines a function
$\rho\colon\Omega\times\Omega\to\mathbb{R}_+$ as follows. For any
$n,m\in\supp\Omega$, $X\in\Omega_n$, $Y\in\Omega_m$, set
$$\rho(X,Y)=\left\|\bigoplus_{\alpha=1}^mX-\bigoplus_{\beta=1}^nY\right\|_{nm}.$$
Since $\mathcal{S}=\mathcal{V}$ is an operator space, $\rho$
extends $\rho_n$ for every $n\in\supp\Omega$. Clearly, $\rho$ is a
pseudometric on $\Omega$, and $\rho(X,Y)=0$ if and only if $X\sim
Y$. Observe that $\widehat{\Omega}$ is a metric space with respect
to the quotient metric $\widehat{\rho}$ defined by
$\widehat{\rho}(\widehat{X},\widehat{Y}):=\rho(X,Y)$. Moreover,
for a nc set $\widetilde{\Omega}\supseteq\Omega$ there is a
natural embedding
$\widehat{\Omega}\hookrightarrow\widehat{\Omega^\sim}$, and
$\widehat{f}$ admits an extension $\widehat{{f}^\sim}$ to
$\widehat{{\Omega}^\sim}$ which also has a unique fixed point, the
equivalence class of a matrix $X_*\in\mathcal{V}^{d\times d}$,
such that $X_{*n}=\bigoplus_{\alpha=1}^{n/d}X_*$ for every
$n\in\supp\Omega$.

 Even though part 2 of Theorem
\ref{thm:thfixedp-2} can be interpreted as a unique fixed point
theorem in a metric space, it does not follow directly from the
classical Banach Fixed Point Theorem for two reasons. First, the
possibility that
 the smallest
possible Lipschitz constant for $\widehat{f}$, that is
$\sup_{n\in\supp\Omega}c_n$,  is equal to $1$, is not ruled out.
Second, the metric space $\widehat{\Omega}$ may be incomplete, as
the following example shows.
\end{remark}
\begin{ex}\label{ex:incopmplete}
Let $\Omega=\mathcal{\mathbb{C}}_{\rm nc}$, with the $(2,2)$
operator norm topology on $\mathbb{C}^{n\times n}$, $n=1,\ldots$.
Define recursively the following sequence in $\Omega$:
$$X^0:=1\in\Omega_1,\quad X^{j+1}:=\begin{bmatrix}
X^j & \frac{1}{2^{j}}I_{2^j}\\
\frac{1}{2^{j}}I_{2^j} & X^j
\end{bmatrix}\in\Omega_{2^{j+1}}.$$
This is a Cauchy sequence in pseudometric $\rho$. Indeed,
$$\rho(X^{j+k},X^j)\le\|X^{j+k}-X^{j+k-1}\|_{2^{j+k}}+\cdots+\|X^{j+1}-X^j\|_{2^{j+1}}=\frac{1}{2^{j+k-1}}+\cdots+\frac{1}{2^j}\to 0$$
 as $j,k\to\infty$. Correspondingly,
the sequence $\{\widehat{X^j}\}_{j=1,\ldots}\in\widehat{\Omega}$
is Cauchy in metric $\widehat{\rho}$:
$$\widehat{\rho}(\widehat{X^{j+k}},\widehat{X^j})=\rho(X^{j+k},X^j)\to
0$$
 as $j,k\to\infty$. Since matrices $X^j$ have infinitely increasing number
 of nonzero
 entries in the first row, and  the sequence $X^{j}_{1k}$ is
 eventually constant for every fixed $k$, and the entries
 $X^{j}_{1k}$ continuously depend on matrices $X^j$, the
 representatives of the class $\widehat{X}:=\lim_{j\to\infty}\widehat{X^j}$
 (provided the limit exists) would have infinitely increasing number
 of nonzero
 entries in the first row as the matrix size of these representatives increases. This is, however,
  impossible for elements
 of $\widehat{\Omega}$, thus the sequence $\{\widehat{X^j}\}_{j=1,\ldots}$
 does not converge, and the metric space $\widehat{\Omega}$ is
 incomplete.
\end{ex}

\begin{ex}\label{ex:5}
In the setting of Example \ref{ex:4}, restrict $f_t$ to the set
$\Omega$ of all matrices of the $(2,2)$ operator norm at most 1.
Then $\supp\Omega =\mathbb{N}$, and the function $f_t$ is a
self-mapping of $\Omega$ for all $t\colon |t|\le 1$. For every
$t\colon |t|<1$, the function $f_t$ satisfies the assumptions of
Theorem \ref{thm:thfixedp-2} (with the metric $\rho_n$ induced by
the $(2,2)$ operator norm  and with
 $c_n=t$ for all $n$), and we obtain the conclusion of the theorem with $X_*=\mathbf{O}_{1\times 1}$ (and since
$d=1$, part 2 of the theorem becomes trivial). The function $f_1$ does not satisfy the contractivity assumption
with any $c_n<1$, and $f_1$ has a plenty of fixed points: every diagonal matrix is a fixed point.
\end{ex}

Our proof of Theorem \ref{thm:thfixedp-1} is based on the
following lemma (whose name is explained in Remark
\ref{rem:quotient-space}).
\begin{lemma}[\textbf{Noncommutative singleton lemma}]\label{lem:nc singleton}
Let $\mathcal{S}$ be a set, $O\in\mathcal{S}$. Let $\Omega
\subseteq \mathcal{S}_{\rm nc}$ respect direct sums of matrices
and be of the form $\Omega=\{X_n\}_{n\in\supp\Omega}$ with
$X_n\in\Omega_n$. Then there exists $X\in \mathcal{S}^{d\times
d}$, with $d=\gcd\{n\colon n\in\supp\Omega\}$, such that
\begin{equation}\label{eq:singleton}
X_n=\bigoplus_{\alpha=1}^{n/d}X,\quad n\in\supp\Omega.
\end{equation}
\end{lemma}

We also apply Lemma \ref{lem:nc singleton} to obtain a nc version
of another important principle of analysis.
\begin{thma}[\textbf{The principle of nested nc sets}]\label{thm:nested}
Let $\mathcal{S}$ be a set, $O\in\mathcal{S}$, and let $\Omega
\subseteq \mathcal{S}_{\rm nc}$ respect direct sums of matrices.
Suppose that $\Omega_{n}$ is a complete metric space with respect
to a metric $\rho_{n}$ for every $n \in \supp{\Omega}.$ Given a
sequence of sets
$$\Omega^j=\coprod_{n\in\supp\Omega}\Omega^j_n\subseteq\Omega,\quad j=1,\ldots$$
such that
\begin{itemize}
\item $\Omega^j$ respects direct sum of matrices, for every
$j=1,\ldots$;
    \item $\Omega^j_n$ is a non-empty closed subset of $\Omega_n$,
     for every $n\in\supp\Omega$,
    $j=1,\ldots$;
    \item $\Omega_1\supseteq\Omega_2\supseteq\cdots$;
    \item $\diam\Omega^j_n:=\sup_{X,Y\in\Omega^j_n}\rho_n(X,Y)\to 0$ as $j\to\infty$, for every
    $n\in\supp\Omega$,
\end{itemize}
there exists a unique $X_*\in \mathcal{S}^{d\times d}$, with
$d=\gcd\{n\colon n\in\supp\Omega\}$, such that
$$\bigcap_{j=1}^\infty\Omega^j=\left\{\bigoplus_{\alpha=1}^{n/d}X_*\right\}_{n\in\supp\Omega}.$$
\end{thma}
\begin{remark}\label{rem:nested}
Similarly to Remark \ref{rem:quotient-space}, we can interpret the
conclusion of Theorem \ref{thm:nested} in terms of quotient sets
of nc sets. Namely, in the assumptions of Theorem
\ref{thm:nested}, the sequence of nested quotient sets
$\widehat{{\Omega}^j}$ has a nonempty intersection consisting of a
single point. Again, the assumptions on the metrics $\rho_n$ are
too general to have a good interpretation in terms of quotient
sets. If we assume, as in part 2 of Theorem \ref{thm:thfixedp-2},
that $\mathcal{S}=\mathcal{V}$ is an operator space and the
metrics $\rho_n$ are norm-induced, then we can define the
corresponding pseudometric $\rho$ on $\Omega$ which extends every
$\rho_n$ and the metric $\widehat{\rho}$ on $\widehat{\Omega}$ as
in Remark \ref{rem:quotient-space}. Clearly, the sequence
$\diam\widehat{{\Omega}^j}:=\sup_{\widehat{X},\widehat{Y}\in\widehat{\Omega^j}}
\widehat{\rho}(\widehat{X},\widehat{Y})$ is non-increasing.
However, it may happen that it does not converge to $0$. Also, the
sets $\widehat{\Omega^j}$ are not necessarily complete metric
spaces. So, both the assumptions of the classical principle of
nested closed sets may fail in our case, as confirmed by the
following example.
\end{remark}
\begin{ex}\label{ex:nested}
Let $\Omega=\mathbb{C}_{\rm nc}$, with the $(2,2)$ operator norm
topology on $\mathbb{C}^{n\times n}$, $n=1,\ldots$. Let
$$\Omega_n^j:=\left\{X\in\mathbb{C}^{n\times n}\colon
0\le\|X\|_n\le\frac{n}{n+j}\right\},\quad n,j=1,\ldots.$$ It is
easy to see that the sets $\Omega_n^j$ satisfy the conditions of
Theorem \ref{thm:nested} and, for a fixed $n$, $0_{n\times n}$ is
the unique common point of the sets $\Omega_n^j$. Define
$$X_n^j:=\frac{n}{n+j}I_n\in\Omega_n^j,\quad Y_n^j:=-X_n^j\in\Omega_n^j\quad n,j=1,\ldots.$$
We have $\|X_n^j-Y_n^j\|_n\to 2$ as $n\to\infty$, hence $\diam\widehat{\Omega^j}=2$ for every $j$, and
$\lim_{j\to\infty}\diam\widehat{\Omega^j}\neq  0$.
 The sequence
$\{X_n^j\}_{n=1,\ldots}$ is Cauchy in pseudometric $\rho$ for every fixed $j$, and so is the sequence
$\{\widehat{X_n^j}\}_{n=1,\ldots}$ in metric $\widehat{\rho}$. However, the latter has no limit in
$\widehat{\Omega^j}$, since such a limit would be a class whose representative are of norm 1, which is
impossible.
\end{ex}

We now present an application of Theorem \ref{thm:thfixedp-2} to
initial value problems for ODEs in nc spaces. The following
theorem is a nc counterpart of (a version of) the existence and
uniqueness theorem for solutions of ODEs in the classical setting
\cite[Theorem 3.7]{HN}.

\begin{thma}\label{thm:ODE}
Let $\mathcal{I}$ be an interval in $\mathbb{R}$ and let $t_0$ be
a point in the interior of $\mathcal{I}$. Let $\mathcal{V}$ be a
(real or complex) operator space and let
$\Xi\subseteq\mathcal{V}_{\rm nc}$ be a nc set, with $\Xi_n$ a
closed subspace of the Banach space $\mathcal{V}^{n\times n}$ for
every $n\in\supp\Xi$. Let $X_{0n}\in\Xi_n$, $n\in\supp\Xi$, and
let $\{X_{0n}\}_{n\in\supp\Xi}$ be a nc set. Suppose that $g\colon
\mathcal{I}\times\Xi\to\Xi$ satisfies the conditions:
\begin{itemize}
\item $g(t,\cdot)$ maps $\Xi_n$ to itself, $n\in\supp\Xi$, and
respects direct sums of matrices,
 for every $t\in\mathcal{I}$;
 \item $g_n:=g|_{\mathcal{I}\times\Xi_n}$ is continuous for every
    $n\in\supp\Xi$;
        \item There is a constant $C>0$ such that
    $$\|g(t,X)-g(t,Y)\|_n\le C\|X-Y\|_n$$ for every $t\in\mathcal{I}$,
    $n\in\supp\Xi$, and $X,Y\in\Xi_n$.
\end{itemize}
Then
\begin{enumerate}
\item[1.] There is a matrix $X_0\in\mathcal{V}^{d\times d}$, with
$d=\gcd\{n\colon n\in\supp\Xi\}$, such that
\begin{equation}\label{eq:initial}
X_{0n}=\bigoplus_{\alpha=1}^{n/d}X_0,\quad n\in\supp\Xi.
\end{equation}
\item[2.] There is a continuously differentiable function
$X_*\colon\mathcal{I}\to\mathcal{V}^{d\times d}$ such that, for
every $n\in\supp\Xi$,
\begin{equation}\label{eq:solution}
X_{*n}=\bigoplus_{\alpha=1}^{n/d}X_*\colon\mathcal{I}\to\Xi_n
\end{equation}
is a unique solution of the initial value problem for the
first-order ODE
\begin{equation}\label{eq:ODE}
\dot{X}=g_n(t,X), \quad X(t_0)=X_{0n}.
\end{equation}
\item[3.] Suppose that, in addition, $g(t,\cdot)$ respects
similarities of matrices, thus is a nc function, for every
$t\in\mathcal{I}$. Then there exist a nc set
$\widetilde{\Xi}(t)\supseteq\Xi$ with
$\supp\widetilde{\Xi}(t)=\mathbb{N}d$, $t\in\mathcal{I}$, so that
for every $n\in\mathbb{N}d$ one has a fiber bundle $\Psi_n$ with
the total space
$$\widetilde{\Xi}_n=\coprod_{t\in\mathcal{I}}\widetilde{\Xi}(t)_n
\subseteq\mathcal{I}\times\mathcal{V}^{n\times n},$$ the base
space $\mathcal{I}$ and the projection
$\pi_n\colon\widetilde{\Xi}_n\to\mathcal{I}$ defined by
$\pi_n\colon\widetilde{X}(t)\mapsto t$; a map $\widetilde{g}$ of
the set $\coprod_{t\in\mathcal{I}}\widetilde{\Xi}(t)$ to itself
such that $\widetilde{g}_n:=\widetilde{g}|_{\widetilde{\Xi}_n}$ is
a continuous bundle endomorphism
 for every
    $n=\mathbb{N}d$, that extends the function $g$ (where we identify all copies of $\Xi_n$ in
    $\widetilde{\Xi}(t)_n$, $t\in\mathcal{I}$);
 and, for every $n\in\mathbb{N}d$, a unique continuously differentiable cross-section of
the fiber bundle $\Psi_n$,
\begin{equation}\label{eq:solution'}
X_{*n}=\bigoplus_{\alpha=1}^{n/d}X_*\colon\mathcal{I}\to\widetilde{\Xi}_n,
\end{equation}
which is a solution of the initial value problem for the
first-order ODE
\begin{equation}\label{eq:ODE'}
\dot{X}=\widetilde{g}_n(t,X), \quad
X(t_0)=\bigoplus_{\alpha=1}^{n/d}X_{0}.
\end{equation}
\end{enumerate}
\end{thma}

\section{The proofs}\label{sec:proofs}
\begin{proof}[Proof of Lemma \ref{lem:nc singleton}]
First we prove that there exist $ k \in \mathbb{N}, n_{1}, \ldots,
n_{k} \in \supp{\Omega}$ such that
\[ d = {\rm gcd} \{ n_{1}, \ldots, n_{k}\}.\] Order elements of $\supp{\Omega}$ increasingly. We have
$$n_{1} \geq
  {\rm gcd}\{n_{1}, n_{2}\} \geq {\rm gcd}\{n_{1}, n_{2}, n_{3}\}\geq \ldots\ (\geq d).$$
 There is at most a finite number of strict inequalities in this chain of inequalities. Let $k$ be the
first integer satisfying
$${\rm gcd}\{n_{1}, \ldots, n_{k}\} = {\rm gcd}\{n_{1}, \ldots, n_{k}, n_{k+1}\} = \ldots .$$
Then
\begin{equation*} {\rm gcd}\{n_{1}, \ldots, n_{k}\} = {\rm gcd}\{n_{1}, \ldots, n_{k}, \ldots\} = {\rm gcd}
\{ n: n \in \supp{\Omega}\} = d
\end{equation*}
because every $n \in \supp{\Omega}$ is divisible by ${\rm gcd}\{
n_{1}, \ldots, n_{k}\}.$

Let $s = {\rm lcm}\{n_{1}, \ldots, n_{k}\}.$ Then $s \in
\supp{\Omega}$  and for every $j = 1, \ldots, k$, the matrix
$\bigoplus_{\alpha = 1}^{s/n_j}X_{n_j}$ is in $\Omega_s$. By the
assumption, this matrix must coincide with $X_s$. Let us use the
convention that the rows and columns of a $n\times n$ matrix $M$
over $\mathcal{S}$ are enumerated from $0$ to $n-1$. Define the
diagonal shift  $S$ which acts on such matrices as follows:
$$(SM)_{ij} = M_{(i-1)\bmod n, (j-1)\bmod n}, \quad  i,j = 0, \ldots, n-1.$$
Clearly, the inverse shift is given by
$$ (S^{-1}M)_{ij} = M_{(i+1)\bmod n, (j+1)\bmod n}, \quad  i,j = 0, \ldots, n-1.$$
Since $X_{s}$ is a block diagonal (in particular, block circulant)
matrix, we have
\begin{equation}\label{eq:fixedp-2} S^{\pm n_{j}}X_{s} = X_{s}. \end{equation}
Since there exist $m_{1}, \ldots, m_{k} \in \mathbb{Z}$ such that
$ d = m_{1}n_{1} + \ldots +m_{k}n_{k} $ (see, e.g., \cite[Problem
1.1]{Vinograd}),  we obtain from \eqref{eq:fixedp-2} that
\begin{equation}\label{eq:fixedp-3} S^{d}X_{s} = X_{s}. \end{equation}
Since $X_{s}$ has the form  \eqref{eq:singleton}, all off-diagonal
$d\times d$ block entries equal $\mathbf{O}_{d\times d}$, and it
follows from \eqref{eq:fixedp-3} that all $d \times d$ block
diagonal entries of $X_{s}$ are equal, say to $X \in
\mathcal{S}^{d \times d}.$ Therefore,
$X_s=\bigoplus_{\alpha=1}^{s/d}X$. Comparing this with
$X_s=\bigoplus_{\beta=1}^{s/n_j}X_{n_j}$, we obtain
$X_{n_j}=\bigoplus_{\gamma=1}^{n_j/d}X$, for every $j=1,\ldots,k$.

For every $n\in\supp\Omega$, we have
$$\bigoplus_{\alpha=1}^{n_1}X_n=X_{nn_1}=\bigoplus_{\beta=1}^nX_{n_1}=
\bigoplus_{\beta=1}^n\bigoplus_{\gamma=1}^{n_1/d}X=\bigoplus_{\alpha=1}^{n_1}\bigoplus_{\delta=1}^{n/d}X
\in\Omega_{nn_1}.$$ Therefore, \eqref{eq:singleton} holds.
\end{proof}

\begin{proof}[Proof of Theorem \ref{thm:thfixedp-1}]
1. Observe that $\{X_{*n}\}_{n\in\supp\Omega}$ is a nc set.
Indeed, since $f$ respects direct sums, for any
$n,m\in\supp\Omega$ one has
$$f(X_{*n}\oplus X_{*m})=f(X_{*n})\oplus f(X_{*m})=X_{*n}\oplus
X_{*m}\in\Omega_{n+m}.$$ Since $X_{*(n+m)}$ is the only fixed point of $f$ in $\Omega_{n+m}$, one must have
$X_{*n}\oplus X_{*m}=X_{*(n+m)}$, so that the set $\{X_{*n}\}_{n\in\supp\Omega}$ respects direct sums of
matrices. By Lemma \ref{lem:nc singleton}, there exists $X_*\in \mathcal{S}^{d\times d}$ such that
\eqref{eq:fixedp} holds.

 2. If $d \in \supp{\Omega}$, then there is nothing to prove: we
just set $\tilde{\Omega} = \Omega, \tilde{f} = f.$

Let  $d \notin \supp{\Omega}.$  We define nc extensions of
$\Omega$ and $f$ as follows. Set $ \tilde{\Omega}_{d} =
\{X_{*}\},$
$$\tilde{\Omega}_{kd}:= \Omega_{kd}\cup  \bigcup _{k',k'':k'+k'' = k}(\tilde{\Omega}_{k'd}\oplus \tilde{\Omega}_{k''d}),
\qquad  k =2,3,\ldots, $$ where $ \Omega_{kd} = \emptyset$ if $ kd \notin \supp{\Omega}.$ By the construction,
$\tilde{\Omega}$ is an nc set.
 Notice that $\tilde{\Omega}$ consists of matrices which are obtained as direct sums of matrices from $\Omega$ and
 copies of $X_{*}$ in every possible order, and that such direct sum decompositions are not necessarily unique.
Let
$$  X = \bigoplus_{\alpha = 1}^{m}X_{\alpha}, \quad Y = \bigoplus_{\beta = 1}^{n}Y_{\beta}$$
for some $m, n \in \mathbb{N},$ $X_{\alpha} \in \Omega_{j_{\alpha}}$, $j_{\alpha} \in \supp{\Omega}$, or
$X_{\alpha} = X_{*}$, and  $Y_{\beta} \in \Omega_{k_{\beta}}$, $k_{\beta} \in \supp{\Omega}$, or $Y_{\beta} =
X_{*}.$
 Define
\[\tilde{f}(  X) := \bigoplus_{\alpha = 1}^{m}\tilde{f}(X_{\alpha})\] where $\tilde{f}(  X_{\alpha}) :=
f(X_{\alpha})$ if $ X_{\alpha} \in \Omega_{j_{\alpha}}$, and $\tilde{f}(  X_{\alpha}) := X_{*}$ if $ X_{\alpha} =
X_{*}.$ Define
\[\tilde{f}(  Y) := \bigoplus_{\beta = 1}^{n}\tilde{f}(Y_{\beta})\] where $\tilde{f}(  Y_{\beta}) := f(Y_{\beta})$ if
$ Y_{\beta} \in \Omega_{k_{\beta}}$, and $\tilde{f}(  Y_{\beta}) := X_{*}$ if $ Y_{\beta} = X_{*}.$ We will show
that $\tilde{f}$ is correctly defined and is a nc function. Suppose we have $SY = XS$ for some matrix \[S \in
\mathcal{R}^{(j_{1} + \cdots +j_{m})d \times (k_{1} + \cdots +k_{n})d}\] where we set $j_{\alpha} = 1$ if
$X_{\alpha} =X_{*},$ and $k_{\beta} = 1$ if $Y_{\beta} =X_{*}.$ We may view $S$ as a $m\times n$ block matrix
with blocks $S_{\alpha\beta} \in \mathcal{R}^{j_{\alpha}d\times{k_{\beta}d}}.$ Then
\[ (SY)_{\alpha\beta} =  S_{\alpha\beta} Y_{\beta} =  X_{\alpha} S_{\alpha\beta} = (XS)_{\alpha\beta}.  \]
We have four cases:

Case 1.  $X_{\alpha} \in \Omega_{j_{\alpha}d}$, $Y_{\beta} \in \Omega_{k_{\beta}d}.$ Since $f$ is a nc function,
we have $ S_{\alpha\beta}f( Y_{\beta}) = f( X_{\alpha}) S_{\alpha\beta},$ i.e. $ S_{\alpha\beta}\tilde{f}(
Y_{\beta})
=\tilde{ f}( X_{\alpha}) S_{\alpha\beta}.$\\

Case 2.  $X_{\alpha} = X_{*}$, $Y_{\beta} \in \Omega_{k_{\beta}d}.$ Then  $ S_{\alpha\beta} Y_{\beta} =
 X_{*} S_{\alpha\beta}$ implies

\begin{multline*}
 \left ( \left [ \begin{array}{c}S_{\alpha\beta}Y_{\beta}\\ \vdots \\ S_{\alpha\beta}Y_{\beta}  \end{array} \right ] =
 \right) \left [ \begin{array}{c}S_{\alpha\beta}\\ \vdots \\ S_{\alpha\beta} \end{array} \right ]Y_{\beta} \\
=
\underbrace{
\left [ \begin{array}{ccc}X_{*}\\& \ddots \\&& X_{*} \end{array} \right ]}_{ k \ {\rm times},\  kd \in \supp{\Omega}}
\left [ \begin{array}{c}S_{\alpha\beta}\\ \vdots \\ S_{\alpha\beta} \end{array} \right ]\left
( = \left [ \begin{array}{c}X_{*}S_{\alpha\beta}\\ \vdots \\X_{*} S_{\alpha\beta} \end{array} \right ]
 \right)
\end{multline*}
 for any $k$ such that $kd \in \supp{\Omega}.$ Then since $ \left [ \begin{array}{ccc}X_{*}\\& \ddots \\&& X_{*} \end{array}
 \right] \in \Omega_{kd}$  is a fixed point of $f |_{\Omega_{kd}}$, we have

\begin{multline*}
 \left ( \left [ \begin{array}{c}S_{\alpha\beta}f(Y_{\beta})\\ \vdots \\ S_{\alpha\beta}f(Y_{\beta})  \end{array}
 \right ] =
 \right) \left [ \begin{array}{c}S_{\alpha\beta}\\ \vdots \\ S_{\alpha\beta} \end{array} \right ]f(Y_{\beta})
 \\ =\left [ \begin{array}{ccc}X_{*}\\& \ddots \\&& X_{*} \end{array} \right ] \left [ \begin{array}{c}S_{\alpha\beta}\\
\vdots \\ S_{\alpha\beta} \end{array} \right ]\left ( = \left [ \begin{array}{c}X_{*}S_{\alpha\beta}\\ \vdots
 \\X_{*} S_{\alpha\beta} \end{array} \right ]
 \right)\end{multline*}
and  $ S_{\alpha\beta}f( Y_{\beta}) = X_{*} S_{\alpha\beta},$ i.e.
$ S_{\alpha\beta}\tilde{f}( Y_{\beta}) =\tilde{ f}( X_{\alpha}) S_{\alpha\beta}.$\\

Case 3.  $X_{\alpha} \in\Omega_{j_{\alpha}d}$, $Y_{\beta} = X_{*}.$ Then $ S_{\alpha\beta} X_{*} =  X_{\alpha}
S_{\alpha\beta}$ implies

\begin{multline*}
 \left ( \left [ \begin{array}{c}S_{\alpha\beta}X_{*} \ldots  S_{\alpha\beta}X_{*}  \end{array} \right ] =
 \right) \left [ \begin{array}{c}S_{\alpha\beta} \ldots  S_{\alpha\beta} \end{array} \right ]
\underbrace{
\left [ \begin{array}{ccc}X_{*}\\& \ddots \\&& X_{*} \end{array} \right ]}_{ k \ {\rm times},\  kd \in
\supp{\Omega}} \\
=X_{\alpha}\left [ \begin{array}{c}S_{\alpha\beta}\ldots  S_{\alpha\beta} \end{array} \right ]\left ( = \left [
\begin{array}{c}X_{\alpha}S_{\alpha\beta}\ldots X_{\alpha} S_{\alpha\beta} \end{array} \right ]
 \right)
\end{multline*}for any $k$ such that $kd \in \supp{\Omega}.$ Then since
$ \left [ \begin{array}{ccc}X_{*}\\& \ddots \\&& X_{*} \end{array} \right] \in \Omega_{kd}$  is a fixed point of
 $f |_{\Omega_{kd}}$ we have
\begin{multline*}
 \left ( \left [ \begin{array}{c}S_{\alpha\beta}X_{*} \ldots  S_{\alpha\beta}X_{*}  \end{array} \right ] =
 \right) \left [ \begin{array}{c}S_{\alpha\beta} \ldots  S_{\alpha\beta} \end{array} \right ]
\left [ \begin{array}{ccc}X_{*}\\& \ddots \\&& X_{*} \end{array} \right ] =\\
f(X_{\alpha})\left [ \begin{array}{c}S_{\alpha\beta}\ldots  S_{\alpha\beta} \end{array} \right ]\left ( =
 \left [ \begin{array}{c}f(X_{\alpha})S_{\alpha\beta}\ldots f( X_{\alpha}) S_{\alpha\beta} \end{array} \right ]
 \right).
\end{multline*}
and  $ S_{\alpha\beta}X_{*} = f(X_{\alpha}) S_{\alpha\beta},$ i.e. $ S_{\alpha\beta}\tilde{f}( Y_{\beta}) =
\tilde{ f}( X_{\alpha}) S_{\alpha\beta}.$\\

Case 4.  If $ X_{\alpha} = X_{*}, Y_{\beta} = X_*{}$, then $S_{\alpha\beta}X_{*} = X_{*}S_{\alpha\beta}$ means
$$S_{\alpha\beta}\tilde{f}(Y_{\beta}) = \tilde{f}(X_{\alpha})S_{\alpha\beta}.$$

Since we have in all these cases that
$$ S_{\alpha\beta}\tilde{f}(Y_{\beta}) = \tilde{f}(X_{\alpha})S_{\alpha\beta}, \quad \alpha = 1, \ldots, m, \
 \beta=1, \ldots, n,$$ we obtain that $S\tilde{f}(Y) = \tilde{f}(X)S.$

In the case where \[X=Y,\quad S = \mathbf{I}_{(j_{1}+ \cdots + j_{m})d}= \mathbf{I}_{(k_{1} + \cdots +k_{n})d},
\] we obtain that $\tilde{f}(X)=\tilde{f}(Y)$, i.e., the definition of $\tilde{f}$ is independent of the
decomposition of a matrix from $ \tilde{ \Omega}$ into a direct sum of matrices from $\Omega$ and copies of
$X_{*}.$ Thus, $\tilde{f}$ is  a correctly defined nc function extending $f$ and, clearly,
$\tilde{f}|_{\tilde{\Omega}_{n}}$ has a unique fixed point
\[  X_{* n} = \bigoplus_{\alpha = 1}^{n/d}X_{*}.\]
\end{proof}

\begin{proof}[Proof of Theorem \ref{thm:thfixedp-2}]

1. By the classical Banach Fixed Point Theorem, a.k.a. the
contractive mapping principle (see, e.g., \cite[pp. 216--217]{RF}
or \cite[Chapter 3]{HN}), for every $n \in\supp{\Omega},$ there
exists a unique $X_{*, n} \in \Omega_{n}$ which is a fixed point
of $f|_{\Omega_{n}}$.
 By Theorem  \ref{thm:thfixedp-1}, there exists $X_{*} \in \mathcal{S}^{d \times
d}$ such that
 $X_{*n} = \bigoplus_{\alpha = 1}^{n/d}X_{*}$ for every $n \in \supp{\Omega}.$

2. Define $ \tilde{\Omega}$ and $\tilde{f}$ as in the proof of part 2 of Theorem \ref{thm:thfixedp-1}. Since for
every $n \in \mathbb{N}, \  \tilde{\Omega}_{nd} = \Omega_{nd}\cup  \bigcup _{n'+n'' =
n}(\tilde{\Omega}_{n'd}\oplus \tilde{\Omega}_{n''d})$ we obtain by induction that $\tilde{\Omega}_{nd}$ is closed
in $\mathcal{V}^{nd \times nd}$ (because the direct sum or the union of a finite number of closed sets is
closed), and therefore is a complete metric space with respect to the metric $\tilde{\rho}_{nd}$ induced by the
norm $\| \cdot \|_{nd}$ on the Banach space $\mathcal{V}^{nd \times nd}.$ We note that a subspace of a complete
metric space is closed if and only if it is complete itself. Next, since every $X \in \tilde{\Omega}$ is a direct
sum of matrices from $\Omega$ and copies of $X_{*}$, and $\tilde{f}$ respects direct
 sums, we have \eqref{eq:rel-contract}
with $$\tilde{c}_{n} = \max_{k\in\supp\Omega\colon k\leq n } c_{k}.$$ The estimate \eqref{eq:conv_rate} is
obtained from \eqref{eq:rel-contract} by iteration, so that the convergence of $X^j$ to $X_{*n}$ follows.
\end{proof}
\begin{proof}[Proof of Theorem \ref{thm:nested}]
For every $n\in\supp\Omega$, the sets $\Omega_n^j$, $j=1,\ldots$,
are nested and satisfy the conditions of the classical theorem on
nested closed sets (see, e.g., \cite[Page 195]{RF}), hence there
exists a unique matrix $X_{*n}\in\bigcap_{j=1}^\infty\Omega_n^j$.
Since $\Omega^j$ are nc sets, so is
$$\bigcap_{j=1}^\infty\Omega^j=\bigcap_{j=1}^\infty\coprod_{n\in\supp\Omega}\Omega_n^j=
\coprod_{n\in\supp\Omega}\bigcap_{j=1}^\infty\Omega_n^j=
\{X_{*n}\}_{n\in\supp\Omega}.$$ By Lemma \ref{lem:nc singleton},
there exists a unique $X_*\in\mathcal{S}^{d\times d}$ such that
$X_{*n}=\bigoplus_{\alpha=1}^{n/d}X_{*}$ for every
$n\in\supp\Omega$, and the conclusion of the theorem follows.
\end{proof}
\begin{proof}[Proof of Theorem \ref{thm:ODE}]
1. This part follows by Lemma \ref{lem:nc singleton}.

2. For every fixed $n\in\supp\Xi$, the initial value problem
\eqref{eq:ODE} can be reformulated as an integral equation
\begin{equation}\label{eq:integral}
X(t)=X_{0n}+\int_{t_0}^tg_n(s,X(s))\,ds.
\end{equation}
By the fundamental theorem of calculus, a continuous solution of
\eqref{eq:integral} is a continuously differentiable solution of
\eqref{eq:ODE}. Fix any $\delta<1/C$. For every $n\in\supp\Xi$,
let $\Omega_n$ be the metric space of continuous functions
$X_n\colon [t_0,t_0+\delta]\to\Xi_n$ with $X_n(t_0)=X_{0n}$ and
with the norm-induced metric. Since $\Xi_n$ is a closed subspace
of the Banach space $\mathcal{V}^{n\times n}$, it is Banach
itself. Since $\Omega_n$ is a closed subset in the Banach space
$C([t_0,t_0+\delta],\Xi_n)$, it is a complete metric space. The
space $\mathcal{W}:=C([t_0,t_0+\delta],\mathcal{V})$ is an
operator space, and for every $n=1,\ldots,$ the space
$\mathcal{W}^{n\times n}$ can be naturally identified with
$C([t_0,t_0+\delta],\mathcal{V}^{n\times n})$. Identifying
$\Omega_n$ with a closed set in $\mathcal{W}^{n\times n}\cong
C([t_0,t_0+\delta],\mathcal{V}^{n\times n})$, we conclude that
$\Omega:=\coprod_{n\in\supp\Xi}\Omega_n$ is a nc set in the nc
space $\mathcal{W}_{\rm nc}$. Define a mapping
$f\colon\Omega\to\Omega$ as follows: for $n\in\supp\Xi$ and
$X\in\Omega_n$, set
$$f(X):=X_{0n}+\int_{t_0}^tg_n(s,X(s))\,ds.$$
Observe that $f$ respects direct sums of matrices. We also have
\begin{multline*}
\|f(X)-f(Y)\|_\infty=\sup_{t\in[t_0,t_0+\delta]}\|[f(X)](t)-[f(Y)](t)\|_n\\
=\sup_{t\in[t_0,t_0+\delta]}
\left\|\int_{t_0}^t\left(g_n(s,X(s))-g_n(s,Y(s))\right)\,ds\right\|_n\\
\le\sup_{t\in[t_0,t_0+\delta]}
\int_{t_0}^t\left\|g_n(s,X(s))-g_n(s,Y(s))\right\|_n\,ds\\
\le\sup_{t\in[t_0,t_0+\delta]}
\int_{t_0}^tC\left\|X(s)-Y(s)\right\|_n\,ds\\
\le C\delta\|X-Y\|_\infty.
\end{multline*}
By Theorem \ref{thm:thfixedp-2}, $f|_{\Omega_n}$ has a unique
fixed point $X_{*n}$ for every $n\in\supp\Xi$. This fixed point is
a continuous solution of \eqref{eq:integral}, and thus a
continuously differentiable solution of \eqref{eq:ODE} on
$[t_0,t_0+\delta]$. Moreover,
$X_{*n}=\bigoplus_{\alpha=1}^{n/d}X_*$ for every $n\in\supp\Xi$,
where $X_*$ is a continuously differentiable function on
$[t_0,t_0+\delta]$ with values in $\mathcal{V}^{d\times d}$. The
same argument applies to $t_0-\delta\le t\le t_0$. The argument
holds for any $t_0\in \mathcal{I}$, and by covering $\mathcal{I}$
with overlapping intervals of length $\delta$, we can extend $X_*$
to a continuously differentiable function on $\mathcal{I}$, so
that \eqref{eq:solution} is a unique solution of \eqref{eq:ODE}.

3. If $d\in\supp\Xi$, we just define $\widetilde{\Xi}(t):=\Xi$, $t\in\mathcal{I}$, so that for every
$n\in\mathbb{N}d$ we have a trivial fiber bundle, with the total space $\mathcal{I}\times\Xi_n$. The initial
value problem \eqref{eq:ODE'} is then identified with \eqref{eq:ODE}, and its solution \eqref{eq:solution'} is
identified with \eqref{eq:solution}.

 If $d\notin\supp{\Xi}$, then applying the argument in our proof of part 2 of this theorem for the
line segment $[t_0,t_0+\delta]$, we can apply  part 2 of Theorem
\ref{thm:thfixedp-2} to the contractive mapping
$f\colon\Omega\to\Omega$ (which respects matrix similarities, as
$g(t, \cdot)$ does for every $t\in\mathcal{I}$) and obtain a nc
set $\widetilde{\Omega}\supseteq\Omega$, with
$\supp\widetilde{\Omega}=\mathbb{N}d$, and a nc function
$\widetilde{f}\colon\widetilde{\Omega}\to\widetilde{\Omega}$ which
extends $f$ and has a unique fixed point,
$X_{*n}=\bigoplus_{\alpha=1}^{n/d}X_*$, in each
$\widetilde{\Omega}_n$, $n\in\mathbb{N}d$. Using our construction
of $\widetilde{\Omega}$ as in the proof of Theorem
\ref{thm:thfixedp-1},
 $$\widetilde{\Omega_d}:=\{X_*\},\quad \tilde{\Omega}_{kd}:=
\Omega_{kd}\cup  \bigcup _{k',k'':k'+k'' =
k}(\tilde{\Omega}_{k'd}\oplus \tilde{\Omega}_{k''d}), \qquad  k
=2,3,\ldots, $$ we obtain that for every $n\in\mathbb{N}d$, the
set $\widetilde{\Omega}_n$ consists of continuous
$\mathcal{V}^{n\times n}$-valued functions on $[t_0,t_0+\delta]$
 that are equal to
$X_{0n}=\bigoplus_{\alpha=1}^{n/d}X_0$ at $t_0$; moreover, for any
$t\in [t_0,t_0+\delta]$, the values of these functions at $t$ lie
in the set $\widetilde{\Xi}(t)_n$ defined recursively by
$$\widetilde{\Xi}(t)_d:=\{X_*(t)\},\quad \tilde{\Xi}(t)_{kd}:=
\Xi_{kd}\cup  \bigcup _{k',k'':k'+k'' =
k}(\tilde{\Xi}(t)_{k'd}\oplus \tilde{\Xi}(t)_{k''d}), \qquad  k
=2,3,\ldots. $$ Clearly, $\widetilde{\Xi}(t)_n$ is a complete
metric space with respect to the norm-induced metric, and
 $\widetilde{\Xi}(t)$ is a nc set which contains $\Xi$, for every
 $t\in[t_0,t_0+\delta]$. Covering $\mathcal{I}$ by the intervals
 of length $\delta$, we can extend this construction of
 $\widetilde{\Xi }(t)$ to all $t\in\mathcal{I}$. Next, for every
$t\in\mathcal{I}$, we define
$\widetilde{g}(t,\cdot)\colon\widetilde{\Xi}(t)\to\widetilde{\Xi}(t)$
as follows. Given an arbitrary $X=\bigoplus_{\alpha=1}^m
X_{\alpha}\in\widetilde{\Xi}(t)$, where $X_\alpha\in\Xi$ or
$X_\alpha=X_*(t)$ (as in the proof of part 2 of Theorem
\ref{thm:thfixedp-1}, every element of $\widetilde{\Xi}(t)$ must
be such a direct sum of matrices), we define
$$\widetilde{g}(t,X)=\bigoplus_{\alpha=1}^m\widetilde{g}(t,X_\alpha),$$
where $\widetilde{g}(t,X_\alpha)=g(t,X_\alpha)$ when
$X_\alpha\in\Xi$, or $\widetilde{g}(t,X_\alpha)=X_*(t)$ when
$X_\alpha=X_*(t)$. As in  the proof of part 2 of Theorem
\ref{thm:thfixedp-1}, we can show  that $\widetilde{g}(t,X)$ is
independent of the representation of $X$ as a direct sum of
matrices and that $\widetilde{g}(t,\cdot)$ is a nc function. The
remaining part of the proof is straightforward, and we leave it to
the reader.
\end{proof}

\end{document}